\documentclass[10pt]{amsart}
\usepackage{amsmath,amsfonts,amsthm,amssymb,amscd}
\def\classification#1{\def\@class{#1}}
\classification{\null}

\DeclareFontFamily{OT1}{rsfs}{}
\DeclareFontShape{OT1}{rsfs}{n}{it}{<-> rsfs10}{}
\DeclareMathAlphabet{\mathscr}{OT1}{rsfs}{n}{it}

\newcommand{\R}{{\mathbb R}}

\newcommand{\Ka}{\mathcal{K}}

\newcommand{\om}{\boldsymbol \omega}
\newcommand{\qu}{\boldsymbol q}
\newcommand{\ve}{\boldsymbol v}
\newcommand{\ur}{\boldsymbol u}

\newtheorem{theorem}{Theorem}

\newtheorem{lemma}{Lemma}

\theoremstyle{remark}

\title{On the use of Klein quadric for geometric incidence problems in two dimensions}

\author{Misha Rudnev}
\address{Misha Rudnev, Department of Mathematics, University of Bristol,
  Bristol BS8 1TW, United Kingdom}
\email{m.rudnev@bristol.ac.uk}

\author{J. M. Selig}
\address{J. M. Selig, Faculty of Business,
London South Bank University,
103 Borough Road,
London SE1 0AA, United Kingdom}
\email{seligjm@lsbu.ac.uk}

\begin{document}
\begin{abstract} We discuss a unified approach to a class of geometric combinatorics incidence problems in two dimensions, of the Erd\H os distance type. The goal is obtaining the second moment estimate. That is, given a finite point set $S$ in $2D$, and a function $f$ on $S\times S$, find the upper bound for the number of solutions of the equation
\begin{equation}\label{gen}
 f(p,p') = f(q,q')\neq 0,\qquad (p,p',q,q')\in S\times S\times S\times S.
\end{equation}
E.g., $f$ is the Euclidean distance in the plane, sphere, or a sheet of the two-sheeted hyperboloid.

Our ultimate tool is the Guth-Katz incidence theorem for lines in $\mathbb{RP}^3$, but we focus on how the original problem in $2D$ gets reduced to its application. The corresponding procedure was initiated by Elekes and Sharir, based on symmetry considerations.  The point we make here is that symmetry considerations can be bypassed or made implicit.  The classical Pl\"ucker-Klein formalism for line geometry enables one to directly interpret a solution of \eqref{gen} as intersection of two lines in $\mathbb{RP}^3$. This allows for a very brief argument extending the Euclidean plane
distance argument to the spherical and hyperbolic distances. We also find many instances of the question \eqref{gen} without underlying symmetry group.

The space of lines in the projective three-space, the Klein quadric $\mathcal K$, is four-dimensional. Thus, we start out with an injective map $\mathfrak F:\,S\times S\to\mathcal K$, that is from a pair of points $(p,q)$ to a line $l_{pq}$ and seek a   corresponding combinatorial problem in the form (\ref{gen}) in two dimensions,  which can be solved by applying the Guth-Katz theorem to the set of lines $\{l_{pq}\}$ in $\mathbb{RP}^3$. 
 
We identify a few new such problems, and hence applications of the Guth-Katz
theorem and make generalisations of the existing ones. It is
the direct approach in question that
is the main purpose of this paper.
\end{abstract}

\maketitle

\vspace{-5mm}

\section{Introduction} In 2010 Guth and Katz, \cite{GK}, settled the long standing Erd\H os distance conjecture. They proved that a set $S$ of $N$ points in $\R^2$ determines $\Omega\left(\frac{N}{\log N}\right)$ distinct Euclidean distances between pairs of points in $S$.

Their proof has two key steps. The first one is to reduce the problem about distances in $2D$ to that of line-line incidences in $3D$. In order to do so, Guth and Katz used what since has become known as the ``Elekes-Sharir framework", presented in \cite{ES}, see also the references contained therein. Given two points $p,q\in S$, consider the set of rotations in the plane that map $p$ to $q$. If $p\neq q$, the centre of such a rotation lies on the bisector to $[pq]$, and the cotangent of the half-angle of rotation $\phi$ changes linearly as one moves along the bisector from the midpoint of $[pq]$. Hence, in the Euclidean coordinates $(x,y,z)$, where $(x,y)$ are the coordinates of the rotation centre and $z=\cot\phi$, the set of plane rotations that take $p=(p_1,p_2)$ to $q=(q_1,q_2)$ is given by a line with the equation
\begin{equation}\label{esline}
l_{pq}:\qquad (x,y,z)(t) = \left( \frac{p_1+q_1}{2}, \frac{p_2+q_2}{2},0\right) + t \left(  \frac{q_2-p_2}{2}, \frac{p_1-q_1}{2},1 \right).
\end{equation}
Besides the translation from $p$ to $q$ (which is irrelevant for the ensuing incidence count at the next step) can be associated with the point at infinity on this line, embedded in the projective space $\mathbb{RP}^3$. It follows that for $p,q,p',q'\in S$,

\begin{equation}\label{intrs}
\|p-p'\|=\|q-q'\|\qquad\Leftrightarrow\qquad l_{pq}\cap l_{p'q'}\neq \emptyset.
\end{equation}

The second key step was a new incidence theorem on line-line intersections in $\R^3$.
\begin{theorem}\label{GKT} Consider a set of $N^2$ lines in $\R^3$, such that

(i)
 no more than $O(N)$ lines are concurrent,

(ii)
 no more than $O(N)$ lines are co-planar,

(iii) no more than $O(N)$ lines lie in a regulus.\footnote{We adhere in this note to the standard terminology in line geometry texts, where
 the term {\em regulus} is used for a single ruling of a doubly-ruled surface.}

Then the number of pairs of intersecting lines is $O\left({N^3}\log N\right).$
\end{theorem}

Once the conditions of Theorem \ref{GKT} have been checked to be satisfied, one gets the ``second moment'' upper estimate $O\left({N^3}\log N\right)$ on the number of pairs of congruent line segments with endpoints in $S$, cf. \eqref{intrs}. The lower bound on the cardinality of the distance set $\Delta(S)$, i.e., number of classes of segments by congruence, follows by the Cauchy-Schwarz inequality:

\begin{equation}\label{csuse} |\Delta(S)| \geq \frac{N^4}{O(N^3\log N)} = \Omega\left(\frac{N}{\log N}\right).\end{equation}

As usual, we use the notation $|\cdot|$ for cardinalities of finite sets. Symbols $\ll$, $\gg,$ suppress absolute constants in inequalities, as well as respectively do the symbols $O$ and $\Omega$. Besides,  $X=\Theta(Y)$ means that $X=O(Y)$ and $X=\Omega(Y)$. The symbols $C$ and $c$ stand for absolute constants, which may change from line to line.

\medskip
A reasonable question appears to be what other two-dimensional geometric combinatorics problems can be treated in terms of Theorem \ref{GKT}. Tao in his blog\footnote{See {\sf terrytao.wordpress.com/2011/03/05/lines-in-the-euclidean-group-se2/.}} stresses the universality of the Elekes-Sharir framework and describes it in the case when $S$ is the point set on the two-sphere $\mathbb S^2$, rather than $\R^2$. In the latter case, he argues that the set of isometries of $\mathbb S^2$ mapping a point $p$ to a point $q$  can be represented by  a great circle on the three-sphere $\mathbb S^3$, which doubly covers the  symmetry  group $SO(3)$. This can be seen by using quaternions. Furthermore, great circles project through the centre of $\mathbb S^3$ as lines in $\R^3$, which can be expected to satisfy the conditions of Theorem \ref{GKT}.  More generally, one can use for the same purpose the Clifford algebra representation of $SO(3)$ itself -- and we spell this out explicitly in the Appendix for comparison with the direct haiku (meaning that it virtually takes three lines) approach in the main body of the paper. Tao also states that in the case of constant negative curvature, that is the hyperbolic plane $\mathbb H^2$ replacing $\mathbb S^2$, the situation must be essentially the same, and in particular one can pass from both corresponding isometry groups $SO(3)$ and  $SL(2),$ to the Euclidean one  $SE(2)$ via the limiting process known as Saletan reduction.

Having felt that there is a certain gap between a blog post and a complete proof, we have decided to furnish one. We do it in essentially three  lines, and without the symmetry argument. 

\medskip 
We then move on to other combinatorial problems in $\R^2$ which can be shown to be amenable to an application of the Guth-Katz theorem.
Roche-Newton and the first author analysed the case of the Minkowski metric in \cite{RR} and found out that owing to the fact that the distance form is sign-indefinite, the hypothesis (ii) of Theorem \ref{GKT} generally gets violated\footnote{In fact, we show below that both hypotheses (i) and (ii) get violated. However, the former hypothesis is violated only at points lying in two planes, which were excluded from the three-space in \cite{RR} by the choice of parameterisation.}. But every line-line incidence inside a plane where the hypothesis was violated was shown to correspond to a zero Minkowski distance. Those could be discounted, once a combinatorial argument to weed the corresponding line intersections in ``rich planes'' out had been developed. This added the symmetry group $SE(1,1)$ to the list of applications of the Elekes-Sharir/Guth-Katz approach.

The incidence estimate of Theorem \ref{GKT} is sharp. Moreover, since the space of lines in $\mathbb{RP}^3$ is four-dimensional, and there are four independent parameters in say \eqref{esline}, the family of lines $\{l_{pq}\}$ arising via the Elekes-Sharir framework can indeed yield an extremal incidence configuration, with the number of lines' pair-wise intersections being $\Theta(N^3\log N)$. This can happen at least in the two cases that have been work out in detail, $SE(2)$ and $SE(1,1)$. What follows upon the application of the Cauchy-Schwarz inequality, cf. \eqref{csuse}, is a different matter, beyond the resolution power of the second moment estimate. E.g., in the case of the Euclidean distance, the omnipresent sharpness example when $S$ is a truncated integer lattice suggests that the ultimate lower bound for the number of distinct distances should be $|\Delta(S)|=\Omega\left(\frac{N}{\sqrt{\log N}}\right)$, a fraction of $\log N$ power better than \eqref{csuse}. In the Minkowski distance case the same example yields $\Theta\left(\frac{N}{(\log N)^\delta ( \log\log N )^{3/2}} \right)$ distinct distances, with $\delta = 0.086071\ldots,$ see \cite{Fo}.
Once again, this is not quite  $|\Delta(S)|=\Omega\left(\frac{N}{\log N}\right)$, as proved in \cite{RR}. On $\mathbb S^2$, there may be no point configurations yielding fewer than $|\Delta(S)|=\Omega({N})$ distances, but we would hesitate to suggest that there are none yielding the logarithmic factor in the second moment estimate. Perhaps,  the explicit expressions \eqref{rpz} for the lines $\{l_{pq}\}$ we provide for the spherical case be useful to furnish a construction of a point set on $\mathbb S^2$ with the extreme value for the second moment if such an example exists.

Whether or not there are point configurations in the $2D$ hyperbolic model $\mathbb H^2$, yielding fewer than 
$\Omega({N})$ distinct hyperbolic distances, appears to be an interesting question, to which we do not know the answer. But in any case, the second moment approach, i.e., counting congruent geodesic segments with endpoints in $S$ is hardly sharp enough to tackle the endpoint issue as to the true minimum number of distinct distances, for one is at the mercy of the application of the Cauchy-Schwarz inequality, \`a-la \eqref{csuse}.

\medskip
All the listed applications of Theorem \ref{GKT} began with the same initial step: symmetry considerations within the Elekes-Sharir framework. In this note we aim to somewhat turn things around and bypass symmetry considerations. We show that one can map directly a point pair $(p,q)\in S\times S$ to a Pl\"ucker vector in the Klein quadric $\mathcal K$, that is the space of lines in $\mathbb{FP}^3$. Thus our main point is simplification of the procedure, which arguably makes it more flexible. We anticipate this to be even more so if one deals with largely open Erd\H os type geometric combinatorics problems in three, rather than two dimensions, in which case the $4D$ "phase space", the Klein quadric in $\mathbb{FP}^5$, gets naturally replaced by the  Study quadric in $\mathbb{FP}^7$ and more generally by a Grassmann manifold..

The field $\mathbb{F}$ for the time being is $\R$, so far as no full extension of Theorem \ref{GKT} to other fields has been established. Still, we often proceed as long as we can  with a general $\mathbb{F}$, since the projective quadric formalism works in a broader context.

\section{Main results}
We re-state the claim that the main point of this note is not so much the novelty of results, but universality and transparency of the method. Our first theorem is the extension of the Guth-Katz Erd\H os distance claim to constant curvature metrics in $2D$.

\begin{theorem}\label{sh} Let $S$ be a set of $N$ points on a $\mathbb S^2$ or $\mathbb H^2$. Then the number of distinct distances between pairs of points of $S$ is $\Omega\left(\frac{N}{\log N}\right)$. \end{theorem}
Once again, our motive, as to Theorem \ref{sh}, formulated in the above-mentioned blog by T. Tao  is to provide a very short proof, bypassing the symmetry argument.

\medskip
It turns out that our viewpoint enables one to identify several types of combinatorial problems where, once the problems are over the reals, the Guth-Katz theorem may be used.  Our next theorem  applies to metric problems and summarises/generalises the Euclidean and Minkowski distance cases as follows.

\begin{theorem}\label{metr} Let $S\subset \R^2$ have $N$ elements and $M_1$, $M_2$ be non-degenerate quadratic forms of the same signature. Then the number of solutions of the equation
\begin{equation}
M_1(p-p') = M_2(q-q')\neq 0,\qquad (p,p',q,q')\in S\times S\times S \times S
\label{dists}\end{equation}
is $O(N^3\log N)$.\end{theorem}
The immediate corollary, in the case of $M_1=M_2=M$, cf. \eqref{csuse}, is the lower bound  $\Omega\left(\frac{N}{\log N}\right)$ on the number of values of
$$M(p-p') \equiv (p-p')^TM(p-p'),$$ unless they are all zero. In the future, we identify the notation for a quadratic form with that for its matrix.

\medskip
The next problem we consider appears to be new.
\begin{theorem}\label{degmetr} Let $S\subset \R^2$ have $N$ elements. Let $(a,c)$, $(\beta,\delta)$ be two pairs of fixed non-collinear - within each pair -- vectors in $\R^2$. Then the number of solutions of the equation
\begin{equation}
(p-p')^T a c^T (p-p') = (q-q')^T \beta \delta^T (q-q') \neq 0,\qquad (p,p',q,q')\in S\times S\times S \times S
\label{degdists}\end{equation}
is $O(N^3\log N)$.\end{theorem}
The immediate corollary, in the case $a=\beta,\,c=\delta$ is the lower bound  $\Omega\left(\frac{N}{\log N}\right)$ on the number of values of $(p-p')^T a c^T (p-p')$, cf. \eqref{csuse}. Or, exclusively, all these values are zero.

Note that as far as Theorem \ref{metr} is concerned, if the signature of the quadratic forms involved is $(1,1)$, each quadratic form $M_1,M_2$ has two isotropic directions $p: M_i(p)=0$. Thus, one can trivially have, say half of the points on an isotropic line for $M_1$ and the other half on an isotropic line for $M_2$. Had the zero value been not excluded, the number of solutions of the equation (\ref{dists}) would have been $\Omega(N^4)$. The same scenario may occur as to Theorem \ref{degmetr}. There, in place of $M_1$ one has a non-symmetric degenerate matrix $ac^T$, which has a left isotropic direction --  orthogonal to $a$, and a right one -- orthogonal to $c$. Theorems \ref{metr} and \ref{degmetr} respectively  imply that in the case $M_1=M_2$ and $ac^T=\beta \delta^T$, either the corresponding quadratic form has
$\Omega\left(\frac{N}{\log N}\right)$  distinct values, evaluated on $p-p'$, or the only value it returns is zero.

\medskip
The other type of problems we identify is counting quadruples of points of $S$ which determine similar directions. More precisely, let $\lambda\neq 0$. For $p,p',q,q'\in S$, with $p=(p_1,p_2)$, and so on, what is the maximum number of solutions of the equation
\begin{equation}
\lambda \,\frac{p_2-p_2'}{p_1-p_1'} \;=\; \frac{q_2-q_2'}{q_1-q_1'}\,?
\label{dirs}\end{equation}
The problem of finding the minimum number of distinct directions, determined by a non-collinear set of points in $\R^2$ was solved, up to the best constant, by the early 1980s. See, e.g., \cite{U}, which has a self-explanatory title {\em 2N Noncollinear points determine at least 2N directions} and the references contained therein.

Proving an upper bound on the number of solutions of (\ref{dirs}) appears to be more involved, and  needs the full power of the Guth-Katz theorem. Note that owing to the possible presence of a single very rich line, supporting, say half of the points, the total number of solutions of (\ref{dirs}) can trivially be $\Omega(N^4)$. So, we have to narrow the point sets $S\subset \mathbb R^2$ in question down to the case when every line supports $O(\sqrt{N})$ points, and then ask for the number of solutions of (\ref{dirs}). The example to bear in mind is again the truncated integer lattice, when the number of solutions of equation \eqref{dirs} with $\lambda=1$ is $\Omega(N^3\log N)$.

\begin{theorem}\label{directions} Let $S\subset \R^2$ have $N$ elements with $O(\sqrt{N})$ points on any straight line.  Then, for any $\lambda\neq 0$, the number of solutions of the equation
(\ref{dirs})
is $O(N^3\log N)$.\end{theorem}

Once again, for $\lambda=1$, Theorem \ref{GKT} provides a sharp bound $O(N^3\log N)$ on the number of solutions of the equation (\ref{dirs}). On the other hand, the logarithmic factor disappears if one asks for the total number of distinct directions. The application of the Cauchy-Schwarz inequality, cf \eqref{csuse}, is to blame for that.

The scopes of Theorems \ref{metr} - \ref{directions} somewhat intersect. E.g., if one takes in Theorem \ref{degmetr} $a=\beta=(0,1)$ and $c=\delta=(1,0)$, then one ends up dealing with Minkowski distances. The same concerns Theorem \ref{directions} in the special case $S=A\times A$, when even somewhat stronger estimates can be obtained via the Szemer\'edi-Trotter theorem, \cite{RN}, \cite{MRS}.

\addtocounter{theorem}{-1}
\renewcommand\thetheorem{\arabic{theorem}$'$}

Theorem \ref{directions} enables quite a far-reaching, in our opinion, generalisation, which gives rise to a whole family of so-called four-variable extractors, that is functions of four variables in a given finite set $A$ of reals, whose range has cardinality $\Omega(|A|^2/\log|A|)$. It follows from Theorem \ref{directions} that, say $$f(a_1,a_2,a_3,a_4)=(a_1-a_2)(a_3-a_4)$$ is such a function, dealing with which, as we mentioned (see \cite{RN}, \cite{MRS}) does not actually need the full might of the Guth-Katz theorem.

However, Theorem \ref{directions} immediately generalises to the following stronger claim (see the following section for background on  Pl\"ucker vectors). 

\begin{theorem}\label{dirG} Let $S\subset \R^2$ have $N$ elements. Consider eight scalar functions $f_1,\ldots,f_4,f'_1,\ldots,f'_4$ on $S\times S$, such that the two sets of $N^2$ lines in $\mathbb R^3$, given by Pl\"ucker vectors
$$\begin{aligned}
\{L_{pq}  &= [f_1:f_2:1:f_3:f_4:-f_1f_3-f_2f_4] (p,q):\,p,q\in S\}, \\ \{L_{p'q'}  & = [f'_1:-f'_2:1:f'_3:f'_4:-f'_1f'_3+f'_2f'_4] (p',q'):\,p',q'\in S\}\end{aligned}
$$
satisfy the conditions of the forthcoming Theorem \ref{GKTP}.

Then the equation \begin{equation}
 [ f_1(p,q)-f_1'(p',q') ]  [ f_3(p,q)-f_3'(p',q') ]  =  [ f_2(p,q)-f_2'(p',q') ]  [ f_4(p,q)-f_4'(p',q') ] :\;\; p,\ldots, q'\in S
\label{genr}\end{equation}
has $O(N^3\log N)$ solutions.
\end{theorem}
Although the statement of Theorem \ref{dirG} is conditional, the reader will see that  checking the conditions of Theorem \ref{GKTP} in lesser generality is routine. 

As a particular case of Theorem \ref{dirG} one can take $f_1, f_1',f_3,f_3' $ as functions of $p$ only and $f_2, f_2',f_4,f_4' $ of $q$ only, equal respectively to $f_1, f_1',f_3,f_3' $ once $q$ replaces $p$. This gives rise to the equation
\begin{equation}
 [ f_1(p)-f_1'(p') ]  [ f_2(p)-f_2'(p') ]  =  [ f_1(q)-f_1'(q') ]  [ f_2(q)-f_2'(q') ].
\label{extr}\end{equation}
In particular, once $S=A\times A$, a Cartesian product, so $p=(a_1,a_2)$ and so on, we expect any ``reasonable''  set of, say four polynomial functions $\{f_1,\ldots,f_2'\}$, satisfy the conditions of the theorem. We believe that specific examples are better off being considered within their specific scope.

\medskip
We end this section by stating the  following slight generalisation of the Guth-Katz theorem, Theorem \ref{GKT}, which is implicit in \cite{RR}. It will be used ``as a hammer'' after the initial set-up procedure in the Klein quadric, the main focus of this paper, has been completed. 
\addtocounter{theorem}{-5}

\begin{theorem}\label{GKTP} Let $L_1,L_2$ be two distinct sets of $N^2$ lines each in $\R^3$, such that

(i)
at any concurrency point there meet no more than $O(N)$ lines from one of the two sets,

(ii)
 no more than $O(N)$ lines from one of the two sets lie in a plane,

(iii) no more than $O(N)$ lines lie in a regulus.

Then the number of intersecting pairs of lines $(l_1,l_2)\in L_1\times L_2$ is $O({N^3}\log N).$
\end{theorem}
\addtocounter{theorem}{6}
\renewcommand\thetheorem{\arabic{theorem}}

\section{Mapping pairs of points to Klein quadric}
In this section we see what happens if one takes a pair of points $(p,q)\in S\times S$ and maps it linearly and injectively to the Klein quadric $\Ka$, thereby defining a line $l_{pq}$ in $\mathbb{FP}^3$. This can be done in many ways. We seek to identify the maps, where one is able to interpret the intersection of $l_{pq}$ with  $l_{p'q'}$  in $\mathbb{FP}^3$ as an instance of the general equation (\ref{gen}). This roughly speaking requires the pairs of variables $(p,q)$, $(p',q')$ corresponding to the lines  $l_{pq}$ and $l_{p'q'}$   to separate into pairs $(p,p')$, $(q,q')$.

\subsection{Background}\label{Bg}
We start with a minimum background which casts, in particular, Conditions (i)-(iii) of Theorem \ref{GKT} in terms of the Klein quadric ${\mathcal K}$. See \cite{JS} for more details.

The space of lines in $\mathbb{FP}^3$ is represented as a projective quadric, known as  the Klein quadric $\mathcal K$ in $\mathbb{FP}^5$, with projective coordinates $(P_{01}:P_{02}:P_{03}:P_{23}:P_{31}:P_{12})$, known as Pl\"ucker coordinates. The line through two points $(q_0:q_1:q_2:q_3)$ and $(u_0:u_1:u_2:u_3)$ in $\mathbb{FP}^3$ has Pl\"ucker coordinates, defined as follows
\begin{equation}
P_{ij}=q_iu_j-q_ju_i.
\label{Pc}\end{equation}
Hence, for a line in $\mathbb{F}^3$, obtained by setting $q_0=u_0=1$, the  Pl\"ucker coordinates acquire the meaning of a projective pair of three-vectors $(\om: \ve)$, where $\om$ is a vector in the direction of the line and for any point $\qu=(q_1,q_2,q_3)$  on the line, $\ve = \qu\times\om$ is the line's moment vector, with respect to some fixed origin. We use the boldface notation for three-vectors throughout.

Conversely, one can denote $\om=(P_{01},P_{02},P_{03}),\; \ve=(P_{23},P_{31},P_{12}),$  the Pl\"ucker coordinates then become $(\om:\ve)$, and treat $\om$ and $\ve$ as vectors in $\mathbb{F}^3$, bearing in mind that, in fact, as a pair they are projective quantities. The lines in the plane at infinity in $\mathbb{FP}^3$ are represented by Pl\"ucker vectors $(\boldsymbol 0:\ve).$ The equation of the Klein quadric ${\mathcal K}$ in $\mathbb{FP}^5$ is
\begin{equation}
P_{01}P_{23}+P_{02}P_{31}+P_{03}P_{12}=0,\;\mbox{ i.e. }\; \om\cdot\ve=0.
\label{Klein}\end{equation}
Equivalently, equation \eqref{Klein} arises after writing out, with the notations \eqref{Pc}, the condition
$$
\det\left(\begin{array}{cccccc} q_0&u_0&q_0&u_0\\q_1&u_1&q_1&u_1\\q_2&u_2&q_2&u_2\\q_3&u_3&q_3&u_3\end{array}\right) = 0.
$$
Two  lines  $l,l'$ in  $\mathbb{FP}^3$, represented by points $L,L'\in \Ka$, with Pl\"ucker coordinates $$L=(P_{01}:P_{02}:P_{03}:P_{23}:P_{31}:P_{12}),\qquad L'=(P'_{01}:P'_{02}:P'_{03}:P'_{23}:P'_{31}:P'_{12})$$ meet  in  $\mathbb{FP}^3$ if and only if
\begin{equation}\label{intersection}
P_{01}P'_{23} +  P_{02}P'_{31} + P_{03}P'_{12} + P'_{01}P_{23} +  P'_{02}P_{31} + P'_{03}P_{12}\;=\;0.\end{equation}
The left-hand side in the elation \eqref{intersection} above is known as the reciprocal product, and can be re-stated as $L^T \mathcal Q L'=0$, where 
$$
\mathcal Q = \left(\begin{array}{ccc} 0 & I_3\\ I_3 & 0\end{array}\right),
$$
where $I_3$ is the $3\times 3$ identity matrix. To avoid confusion we use the lowercase notation for lines $l$ in $\mathbb{FP}^3$; they are represented by points $L\in \Ka$, the uppercase notation.

If the Pl\"ucker coordinates of the two lines are written as $L=(\om:\ve)$ and  $L'=(\om':\ve')$, then  the zero reciprocal product condition can be expressed as
\begin{equation}\label{intersectionv}
\om\cdot \ve'+ \ve\cdot \om' = 0.\end{equation}

Using the latter three equations, it is easy to see, by taking the gradient of \eqref{Klein} that a $\mathbb{FP}^4$ in $\mathbb{FP}^5$ is tangent to $\Ka$ at some point $L$ if an only if the corresponding dual vector, defining the hyperplane is itself in the Klein quadric in $\mathbb{FP}^{5*}$. Moreover, it follows from (\ref{intersection}) that $T_L \Ka\cap \Ka$ consists of $L'\in \Ka$, representing all lines $l'$ in $\mathbb{FP}^3$, incident to the line $l$. This set of lines is usually called a singular line complex.

\medskip
The largest dimension of a projective subspace contained in $\Ka$ is two. Copies of $\mathbb{FP}^2$ contained in $\Ka$ have important meaning which we describe next. To this end, $\Ka$ has two  (assuming ${\rm char}(\mathbb F)\neq 2$) ``rulings'' by planes, which lie entirely in the quadric, with the fibre space of each ruling being $\mathbb{FP}^3$. The other important type of subvarieties in  $\Ka$, relevant to the subject of this note are conics, arising as transverse intersections of $\Ka$ with two-planes.

The Klein quadric contains a three-dimensional family of projective two-planes, called $\alpha$-planes. Elements of a single $\alpha$-plane are lines, concurrent at some point $(q_0:q_1:q_2:q_3)\in \mathbb{FP}^3$. If  the concurrency point is $(1:\qu)$, which is identified with $\qu\in \mathbb F^3$,  the $\alpha$-plane is a graph $\ve = \qu\times \om$. Otherwise, an ideal concurrency point  $(0:\om)$ gets identified with some fixed $\om$, viewed as a  projective vector. The corresponding $\alpha$-plane is the union of the set of parallel lines in $\mathbb{F}^3$ in the direction of $\om$, with Pl\"ucker coordinates $(\om:\ve)$, so $\ve\cdot\om=0,$ by \eqref{Klein}, and the set of lines in the plane at infinity incident to the ideal point $(0:\om)$. The latter lines have Pl\"ucker coordinates $(\boldsymbol 0:\ve),$ with once again $\ve\cdot\om=0$.

Similarly, the Klein quadric contains another three-dimensional family of two-planes, called $\beta$-planes, which represent co-planar lines in $\mathbb{FP}^3$.
A ``generic'' $\beta$-plane is a graph $\om = \ur\times \ve$, for some $\ur \in \mathbb F^3$. The case $\ur =\boldsymbol 0$ corresponds to the plane at infinity, otherwise the equation of the co-planarity plane in $\mathbb{F}^3$ becomes
\begin{equation}\ur\cdot \qu =-1.\label{refer}\end{equation}
If $\ur$ gets replaced by a fixed ideal dual point $(0:\ve)$, the corresponding $\beta$-plane comprises lines, coplanar in planes through the origin: $\ve\cdot \qu = 0$. The corresponding $\beta$-plane in the Klein quadric is formed by the set of lines with Pl\"ucker coordinates $(\om:\ve)$, plus the set of lines through the origin in the co-planarity plane. The latter lines have Pl\"ucker coordinates $(\om:\boldsymbol 0)$. In both cases one requires $\om\cdot\ve = 0$.

Two planes of the same ruling of $\Ka$ always meet at a point, which is the line defined by the two concurrency points in the case of $\alpha$-planes. A $\alpha$- and a $\beta$-plane typically do not meet; if they do this means that the concurrency point, defining the $\alpha$-plane lives in the plane $\pi$, defining the $\beta$-plane.
The intersection is then a straight line, a copy of $\mathbb{FP}^1$ in ${\mathcal K}$, representing a {\em plane pencil of lines} -- the lines in $\pi$ via the concurrency point. These are lines in $\mathbb{FP}^3$, which are co-planar in $\pi$ and concurrent at the concurrency point. Conversely, each line in ${\mathcal K}$ identifies the pair ($\alpha$-plane, $\beta$-plane), that is the plane pencil of lines uniquely. Moreover points $L,L'\in \Ka$ can be connected by a straight line in $\Ka$ if and only if the corresponding lines $l,l'$ in $\mathbb{FP}^3$ meet, cf. \eqref{intersection}.

\medskip
These $\alpha$- and $\beta$-planes represent a specific case when a subspace $\Pi= \mathbb{FP}^2$ of $\mathbb{FP}^5$ is contained in $\Ka$. A semi-degenerate case is when the two-subspace $\Pi$ contains a line in $\Ka$. The non-degenerate situation would be the two-plane intersecting $\Ka$ along a conic. If the field $\mathbb{F}$ is algebraically closed, then any $\Pi$ intersects $\Ka$. Otherwise this is not necessarily the case, take e.g the case when $\Pi$ is defined by the condition $\om =\ve$ for $\mathbb F=\mathbb R$.  

Assume that the equations of the two-plane $\Pi$ can be written as
$$
A\om + B\ve = \boldsymbol 0,
$$
where $A,B$ are some $3\times3$ matrices. How can one describe the union in $\mathbb{FP}^3$ of lines represented by $\Pi\cap\Ka$? For points in $\Pi\cap\Ka$, which do not represent lines in the plane at infinity in $\mathbb{FP}^3$, we can write $\ve =\qu\times \om$, where $\qu$ is some point in $\mathbb{F}^3$, on  the line  with Pl\"ucker coordinates $(\om:\ve)$, and $\om\neq \boldsymbol 0$. If $Q$  is the skew-symmetric matrix $ad(\qu)$ (that is the cross product of $\qu$  with a vector is $Q$ times this vector as a column-vector) we obtain
$$
(A-BQ)\om =\boldsymbol 0\qquad \Rightarrow\qquad \det(A-BQ)=0.
$$
This a quadratic equation in $\qu$, since $Q$ is a $3\times 3$ skew-symmetric matrix, so $\det Q=0$.  If the above equation has a  linear factor in $\qu$, defining a plane in $\mathbb{FP}^3$, then $\Pi\cap \Ka$ contains a line, which represents a pencil of lines in the latter plane in $\mathbb{FP}^3$. If the above quadratic polynomial in $\qu$ is irreducible, then  if  the field $\mathbb{F}$ is algebraically closed we always get a quadric surface in $\mathbb{FP}^3$. This is the precisely the non-degenerate intersection case, when $\Ka\cap \Pi$ is a conic.

In the latte case the two-plane $\Pi$  in $\mathbb{FP}^5$ can be obtained as the intersection of three four-planes, tangent to $\Ka$ at some three points $L_1,L_2,L_3$, corresponding to three mutually skew lines in $\mathbb{FP}^3$. Thus the intersection is a regulus: the set of all lines in $\mathbb{FP}^3$, meeting three given mutually skew lines $l_1,l_2,l_3$.
\label{BG}

\subsubsection{Proof of Theorem \ref{sh}}
We now move on to proofs of our main results. For motivation, let us first show how the Elekes-Sharir symmetry argument can be bypassed if one deals with the second moment estimate for plane Euclidean distances. Rewrite the equation (\ref{intrs}) as
$$
p\cdot p' - q\cdot q' - (\|p\|^2+\|p'\|^2 - \|q\|^2-\|q'\|^2) = 0.
$$
(For two-vectors we do not use the boldface notation.)
The latter equation (cf. \eqref{esline} and \eqref{intersection}) is the condition of the zero reciprocal product of two points $L_{pq}$ and $L_{p'q'}$ in the Klein quadric, with the Pl\"ucker coordinates
\begin{equation}\label{rp}
L_{pq}= \left[\frac{q_2-p_2}{2}: \frac{p_1-q_1}{2}: 1: \frac{p_2+q_2}{2}:  - \frac{p_1+q_1}{2}:  \frac{\|p\|^2-\|q\|^2}{4}\right],
\end{equation}
the same with prime indices for $L_{p'q'}$, where $p=(p_1,p_2)$, etc. The above expression just the Pl\"ucker coordinate expression for the line, given by equation \eqref{esline}. Indeed, the first three Pl\"ucker coordinates are the line's direction vector $\om$, the remaining three are the cross product of the point $\qu= \left(\frac{p_1+q_1}{2},\frac{p_2+q_2}{2} ,0\right)$ on the line with $\om$.

Hence, estimating the number of solutions of \eqref{intrs} is tantamount to estimating the number of pairwise intersections of the lines $\{l_{pq}\}_{(p,q)\in S\times S}$. The fact that this set of lines satisfies the hypotheses of Theorem \ref{GKT} is verified in \cite{GK}; we will shortly do this as to the lines arising in the context of Theorem \ref{sh}.

We now prove Theorem \ref{sh} in the case when the point set  $S$ is supported on the two-sphere $\mathbb S^2$ and for the hyperbolic model. Along the lines of Section \ref{BG} we use boldface notation for three-vectors, except in the notations for the lines $l_{pq}$.

\begin{proof}
Let $S\subset \mathbb S^2$ be the set of $N$ points. Let $\boldsymbol p=(p_1,p_2,p_3)$ be the Euclidean coordinates of $\boldsymbol p\in S$. Clearly, the geodesic segment congruency condition \eqref{intrs} the distance now being the restriction of the Euclidean distance on $\mathbb S^2$ rewrites as
\begin{equation}\label{prz}
\boldsymbol p\cdot \boldsymbol p' = \boldsymbol q\cdot \boldsymbol q',
\end{equation}
where $\cdot$ is the dot product in $\R^3$ (the distance between $\boldsymbol p$ and $\boldsymbol p'$ on the unit sphere being $\arccos(\boldsymbol p\cdot \boldsymbol p' ).$)

The latter can be rewritten as
\begin{equation}\label{prz1}(\boldsymbol p+\boldsymbol q)\cdot(\boldsymbol p'-\boldsymbol q') + (\boldsymbol p'+\boldsymbol q')\cdot(\boldsymbol p-\boldsymbol q) = 0. \end{equation}

Now, the left-hand side is the reciprocal product of two Plucker vectors $L_{pq}$ and $L_{p'q'}$, where
\begin{equation}\label{rpz}
L_{pq} = (\boldsymbol p+\boldsymbol q: \boldsymbol p-\boldsymbol q)= [p_1+q_1: p_2+q_2: p_3+q_3: p_1-q_1: p_2-q_2: p_3-q_3],
\end{equation}
similarly for $L_{p'q'}$.  (We do not use the boldface notations for the subscripts in $L_{pq}$, for the set $S$ is two-dimensional.) Observe that the reciprocal product of $L_{pq}$ with itself equals $\|\boldsymbol p\|^2-\|\boldsymbol q\|^2$, which is zero for any $\boldsymbol p, \boldsymbol q\in S$.

Let us verify that the corresponding set of lines $\{l_{pq}\}_{(p,q)\in S\times S}$ satisfies the hypotheses of Theorem \ref{GKT}. Consider the hypothesis (i). Assuming concurrency at some point in $\mathbb R^3$, there is $\boldsymbol u=(u_1,u_2,u_3)$, such that
\begin{equation}\label{rpzm}
\boldsymbol u\times (\boldsymbol p+\boldsymbol q) = (\boldsymbol p-\boldsymbol q).
\end{equation}
Let $U$ be the skew-symmetric matrix $ad(\boldsymbol u)$, thus, with $I$ for the $3\times 3$ identity matrix, we have
$$
(U-I)\boldsymbol p = -(U+I)\boldsymbol q.
$$
Both matrices in brackets is non-degenerate, and therefore for every $\boldsymbol p$ we have at most one $\boldsymbol q$, satisfying the latter equation. If concurrency occurs at a point at infinity, this fixes $\boldsymbol p+\boldsymbol q$, hence the same conclusion.

Therefore the hypothesis (i) is satisfied: for every concurrency point $\boldsymbol u$, there is at most one line $l_{pq}$ passing through it, for each fixed $\boldsymbol p$.  The verification of (ii) is exactly the same, for now one repeats the argument as to $ (\boldsymbol p+\boldsymbol q) = \boldsymbol u\times(\boldsymbol p-\boldsymbol q). $

Finally, to verify the hypothesis (iii) we refer the reader to the forthcoming Lemma \ref{reguli}, which does it in a fairly general context.

\medskip
In the case of the hyperbolic plane $\mathbb H^2$, by analogue with the above, we take the Hyperboloid model of the hyperbolic metric instead (which is isometric to other models, say $\mathbb H^2$ or the Poincar\'e disk, see e.g. \cite{CFKP}.) I.e. let $\mathbb L\subset \R^3$ (which in the literature stands, apparently, for ``Loid'', \cite{CFKP})  have equation
$$
x_1^2 + x_2^2 -x_3^2 =-1, \qquad x_3>0,
$$
and $S\subset \mathbb L$. The hyperbolic distance between $\boldsymbol p,\boldsymbol p'\in \mathbb L$ equals $\cosh^{-1}(\boldsymbol p\cdot \boldsymbol p')= p_3p'_3-p_1p'_1-p_2p'_2$, that is now (and only through the rest of this proof) $\cdot$ stands for the Minkowski dot product. So the geodesic segment congruency condition \eqref{intrs}, the distance now being the restriction of the Euclidean distance on $\mathbb L$, is given again by \eqref{prz}, \eqref{prz1} only in terms of the Minkowski dot product. 

Now, the left-hand side of  \eqref{prz1}  is the reciprocal product of two Plucker vectors $L_{pq}$ and $L_{p'q'}$, where
\begin{equation}\label{rpzh}
L_{pq} = [p_1+q_1: p_2+q_2:p_3+q_3:p_1-q_1:p_2-q_2: -(p_3-q_3)],
\end{equation}
similarly for $L_{p'q'}$. Thus, the only difference so far with the case $S\subset \mathbb S^2$ is the sign change of the last component of the Pl\"ucker 6-tuple. Observe that the reciprocal product of $L_{pq}$ with itself equals zero for any $\boldsymbol p,\boldsymbol q\in \mathbb L$.

To verify the hypothesis (i) of Theorem \ref{GKT} one now has the analogue of (\ref{rpzm}), with the matrix $D={\rm diaq}(1,1,-1)$ as follows:
$$
D\boldsymbol p-\boldsymbol  u\times \boldsymbol p  = D\boldsymbol q+ \boldsymbol u\times \boldsymbol q.
$$
The condition $D\boldsymbol p- \boldsymbol u\times \boldsymbol p = 0$, which is necessary for having more than one one line $l_{pq}$ passing through the concurrency point $\boldsymbol u$ for each fixed $\boldsymbol q$, means that  $\boldsymbol p$ is such that its reflection w.r.t. the $(x_1x_2)$-plane is tantamount to vector multiplication by $\boldsymbol u$. This is only possible when $u_3=0$ and $\boldsymbol p$ lie on the light cone $x_1^2+x_2^2-x_3^2=0,$ but not on $\mathbb L$. The same conclusion holds for concurrency at infinity, which requires that  $\boldsymbol p+\boldsymbol q$ be fixed. Hence, as long as $\boldsymbol p,\boldsymbol q\in \mathbb L$, the hypothesis (i) of Theorem \ref{GKT} is satisfied. The same argument applies to the hypothesis (ii). To verify the hypothesis (iii) the reader is referred to the forthcoming Lemma \ref{reguli}.

\end{proof}

\subsection{Map $\mathfrak F$ and separating variables}
From now on, till the Appendix, we deal with the plane set $S\in \mathbb F^2$ (to apply the Guth-Katz theorem one must have $\mathbb F=\mathbb R$). We now consider linear maps of $(p,q)\in S\times S$ to $\Ka$ as follows.
Let $a,\alpha, b,\beta, c,\gamma, d,\delta \in \mathbb F^2$. Let
\begin{equation}\label{forms}\begin{aligned} L_1(p,q) = a\cdot p + \alpha\cdot q, \qquad  L_2(p,q) = b\cdot p + \beta\cdot q,\\
L_3(p,q) = c\cdot p + \gamma\cdot q, \qquad  L_4(p,q) = d\cdot p + \delta\cdot q.
\end{aligned}
\end{equation}
(We hope that our re-use of the symbols $\alpha,\beta$ as two-vectors will not cause confusion: in Section \ref{Bg} we defined $\alpha$- and $\beta$-planes in the Klein quadric $\Ka$.)

Map $(p,q)\to \Ka$ as follows:
\begin{equation}\label{map}\begin{aligned}
\mathfrak F:\;(p,q) & \to L_{pq}\\
&=[L_1(p,q):L_2(p,q):1:L_3(p,q):L_4(p,q):- L_1(p,q)L_3(p,q) - L_2(p,q)L_4(p,q)],\end{aligned}
\end{equation}
where the right-hand side is Pl\"ucker coordinates.  The linear forms $L_1,\ldots, L_4$ should be linearly independent to ensure injectivity of the assignment. We shall make explicit checks with the particular choices, in the context of Theorems \ref{metr}-\ref{directions}.

Linear independence of the linear forms $L_1,\ldots, L_4$ alone ensures that Condition (iii) of Theorem \ref{GKT} is satisfied. What follows is an easy generalisation of Lemma 2.9 in \cite{GK}.

\begin{lemma} If the linear forms $L_1,\ldots, L_4$ are linearly independent,  Condition (iii) of Theorem \ref{GKT} is satisfied for the family of lines $\{l_{pq}\}$ defined  by (\ref{map}). It is also satisfied for the families $\{l_{pq}\}$ defined by \eqref{rpz}, \eqref{rpzh} in the context of Theorem \ref{sh}. \label{reguli}\end{lemma}

\begin{proof} If the linear forms $L_1,\ldots, L_4$ are linearly independent, the map (\ref{map}) is injective. Fix $p$ and treat $q$ as a variable in (\ref{forms}). This is also the case with the maps $S\times S\to\mathcal K$, defined by \eqref{rpz}, \eqref{rpzh}. For each $p$, the map from $q\to
\mathcal K$ has full rank. Hence, its image in $\Ka$ is the intersection of $\Ka$ with a projective subspace $\mathbb{FP}^3$ in $\mathbb{FP}^5$, which is called in line geometry literature a linear  {\em congruence}\footnote{In fact, the intersection is transverse, in which case this is a {\em linear elliptic congruence}, a two-dimensional family of pair-wise skew lines. This is a well known figure with many interesting properties, for example it can be viewed as a set of reguli on concentric hyperboloids, see e.g., \cite{PW} .}.

A regulus, that is a conic, arising as the transverse intersection of
$\Ka$ with a $\mathbb{FP}^2$, will be either contained in the above congruence or intersect it at most two points.

Hence, given a regulus in $\mathbb{FP}^3$, it is either contained in the family of lines $\{l_{pq}\}_{q\in \mathbb F^2}$ for a fixed $p$ or has at most two lines with the latter set in common. It follows that the maximum number of lines from the finite collection $\{l_{pq}\}$ that can lie in a regulus is $2N$, and therefore at most $4N$ in a doubly-ruled surface in $\mathbb{FP}^3$.
\end{proof}

Two lines $l_{pq}$ and $l_{p'q'}$ in $\mathbb{FP}^3$ defined by \eqref{map} intersect in $\mathbb{FP}^3$ if and only if the zero reciprocal product condition \eqref{intersection} is satisfied. I.e.:
\begin{equation}
\left( L_1(p,q) - L_1(p',q'), L_2(p,q) - L_2(p',q') \right) \cdot \left( L_3(p,q) - L_3(p',q'), L_4(p,q) - L_4(p',q') \right)
 = 0.
\label{nonlin}\end{equation}
Hence (here we use linearity of $L$'s)
$$
L_1(p-p',q-q')L_3(p-p', q-q') + L_2(p-p',q-q')L_4(p-p', q-q') = 0.
$$
This means, in view of (\ref{forms}), introducing three $2\times 2$ matrices
\begin{equation}
M_1 = ac^T + bd^T, \qquad M_2  = -(\alpha \gamma ^T + \beta \delta^T), \qquad M_3 = a\gamma^T + b\delta^T + c\alpha^T+ d\beta^T,
\label{matrices}\end{equation}
that
\begin{equation}
 (p-p')^T M_1(p-p') - (q-q')^TM_2 (q-q') + (p-p')^T M_3(q-q') = 0.
\label{matricesmore}\end{equation}

Our goal is to be able to separate variables in (\ref{matricesmore}), that is to be able to rewrite it as $f(p,p')=g(q,q')$, for some functions $f,g$. There are several cases to consider.

\medskip
Variables will separate if  $M_3=0$ or otherwise possibly when $M_1=M_2=0$.

The condition $M_3=0$ means that
\begin{equation}\label{m3}\left(
    \begin{array}{cccc}
      a_1 & b_1 & c_1 & d_1 \\
      a_2 & b_2 & c_2 & d_2 \\
    \end{array}
  \right)
\left(
  \begin{array}{cc}
    \gamma_1 & \gamma_2 \\
    \delta_1 & \delta_2 \\
    \alpha_1 & \alpha_2 \\
    \beta_1 & \beta_2 \\
  \end{array}
\right) = 0.
\end{equation}
Thus, two pairs of four-vectors $(a_1,b_1,c_1,d_1), (a_2,b_2,c_2,d_2)$ and $(\gamma_1, \delta_1,\alpha_1,\beta_1),(\gamma_2, \delta_2,\alpha_2,\beta_2)$  lie in mutually orthogonal two-spaces in $\mathbb F^4$.

There are three cases to consider in this context as far as the matrices $M_1$ and $M_2$ in (\ref{matricesmore}) are concerned. The first two cases arise in the context of Theorem \ref{metr}. They are: when both $M_1,M_2$ are symmetric positive definite (if $\mathbb F=\mathbb R$ or more generally if $-1$ is not a square in $\mathbb{F}$), and when they are both symmetric signature $(1,1)$ (if $\mathbb F=\mathbb R$ or more generally, $-1$ is a square in $\mathbb{F}$). The third case arises in the context of Theorem \ref{degmetr}: the matrices $M_1,M_2$ are non-symmetric degenerate.

Finally, if $M_3\neq 0$, there will be an additional case when $M_1=M_2=0$ and $M_3$ either diagonal or has zeroes on the main diagonal. This is the subject of Theorem \ref{directions}.

\subsubsection{Positive definite metric case}
Take nonzero\begin{equation}\label{eset}
c=a,\;d=b,\;\gamma=-\alpha,\;\delta=-\beta;\qquad  a\neq \lambda b,\;\alpha\neq\lambda \beta,\mbox{ for }\lambda\in \mathbb F.\end{equation}
Clearly, \eqref{m3} is thus satisfied, and we get from \eqref{matricesmore}:

\begin{equation}\label{eucl}\begin{array}{c}
M_1  = \left(
        \begin{array}{cc}
          a_1^2 +  b_1^2 & a_1a_2 + b_1b_2\\
          a_1a_2 + b_1b_2 & a_2^2 +  b_2^2 \\
        \end{array}
      \right), \;\;\;\; M_2 = \left(
        \begin{array}{cc}
          \alpha_1^2 +  \beta_1^2 & \alpha_1\alpha_2 + \beta_1\beta_2\\
          \alpha_1\alpha_2 + \beta_1\beta_2 & \alpha_2^2 + \beta_2^2 \\
        \end{array}
      \right), \\ \hfill \\
 (p-p')^T M_1(p-p')\; \;= \;\;(q-q')^T M_2 (q-q').
\end{array}\end{equation}
The matrices $M_1$, $M_2$ are symmetric positive definite and generalise the case of the Euclidean distance considered in \cite{GK}. Note that this case differs in a general $\mathbb{F}$ from the next one only if $-1$ is not a square in $\mathbb{F}$.

\subsubsection{Signature $(1,1)$ metric case} To generalise the  case of the Minkowski distance considered in \cite{RR}, take nonzero
\begin{equation}\label{mset}
c=a,\; d = -b, \; \gamma = -\alpha,\;\delta=\beta;\qquad a\neq \lambda b,\;\alpha\neq\lambda \beta,\mbox{ for }\lambda\in \mathbb F.
\end{equation}

Then the variables  in \eqref{matricesmore} separate as follows:
\begin{equation}\label{mink}\begin{array}{c}
M_1  = \left(
        \begin{array}{cc}
          a_1^2 -  b_1^2 & a_1a_2 - b_1b_2\\
          a_1a_2 - b_1b_2 & a_2^2 -  b_2^2 \\
        \end{array}
      \right), \;\;\;\; M_2 = \left(
        \begin{array}{cc}
          \alpha_1^2 -  \beta_1^2 & \alpha_1\alpha_2 - \beta_1\beta_2\\
          \alpha_1\alpha_2 - \beta_1\beta_2 & \alpha_2^2 -  \beta_2^2 \\
        \end{array}
      \right), \\ \hfill \\
 (p-p')^T M_1(p-p')\; \;= \;\;(q-q')^T M_2 (q-q').
\end{array}\end{equation}
The matrices $M_1$, $M_2$ are symmetric non-degenerate, with signature $(1,1)$, thus generalising the Minkowski distance in the case $\mathbb F=\mathbb R$.

To this end, let us calculate the ''light cone'' isotropic directions for the matrices $M_1$, $M_2$.

\begin{lemma} $x=(-(a_2\pm b_2), a_1\pm b_1)$ are isotropic vectors for $M_1$, that is $x^TM_1x=0$. Similarly, $x=(-(\alpha_2\pm \beta_2), \alpha_1\pm \beta_1)$ are isotropic vectors for $M_2$.\label{iso}\end{lemma}

\begin{proof} The verification is a brute force calculation.\end{proof}

\subsubsection{Degenerate case} Take
\begin{equation}\label{dset}
b=d=\alpha=\gamma=0,\;\mbox{ and nonzero }\; a\neq \lambda c;\;\beta\neq\lambda \delta,\mbox{ for }\lambda\in \mathbb F.
\end{equation}
Then the variables in  \eqref{matricesmore} separate as in the last line of \eqref{eucl}, \eqref{mink}:\begin{equation}\label{dc}
M_1 = \left(
        \begin{array}{cc}
          a_1c_1 & a_1c_2\\
          a_2c_1 & a_2c_2 \\
        \end{array}
      \right),
      \qquad
      M_2 = -\left(
        \begin{array}{cc}
          \beta_1\delta_1 & \beta_1\delta_2 \\
          \beta_2\delta_1  & \beta_2\delta_2  \\
        \end{array}
      \right).
\end{equation}
In the formulation of Theorem \ref{degmetr} we've changed $\beta\to-\beta.$

Clearly, $y^TM_1 x=(y\cdot a)(x\cdot c)$, and hence will be zero if and only if either $y$ is orthogonal to $a$ or $x$ is orthogonal to $c$. This defines the left and right kernels for $M_1$, and similarly for $M_2$.

\subsubsection{Directions' case} Variables in \eqref{matricesmore} also separate in the special case when $M_1=M_2=0$, and $M_3$ is diagonal or has zeroes on the main diagonal. We consider the latter situation and set, for some $\lambda\neq 0,$
\begin{equation}\label{dirset}
c=d=\alpha=\beta=0,\; \gamma_1= \lambda b_1,\; \delta_1 = -\lambda a_1,\;\gamma_2= b_2, \; \delta_2 =-a_2,\;a_2b_1-a_1b_2\neq 0.\end{equation}
Hence, we have
\begin{equation}\label{direct}\begin{array}{c}
\lambda_1 = a_2b_1-a_1b_2,\qquad \lambda_2= \lambda\lambda_1,\\ \hfill \\
M_3= \left(
       \begin{array}{cc}
         0&-\lambda_1  \\
         \lambda_2&0 \\
       \end{array}
     \right),
\\ \hfill \\
\lambda \frac{p_2-p_2'}{p_1-p_1'} \;\;= \;\;\frac{q_2-q_2'}{q_1-q_1'}.\end{array}\end{equation}
This generalises the problem of counting pairs of points in $S$, which lie on some line in a given direction, then summing over directions.

\section{Proof of Theorems \ref{metr}-\ref{dirG}}
Theorem \ref{dirG} requires no proof once we observe that the equation \eqref{nonlin} above has not used anything about the quantities $L_1(p,q),\ldots,L_4(p',q')$, except that they are scalar functions and we replace them with $f_1,\ldots,f_4'$ (adjusting the signs) appearing in the statement of Theorem \ref{dirG}. In particular the easy calculation that led to \eqref{nonlin} enables that $L_i'= L_i(p',q'), \;i=1,\ldots,4,$ be different functions from $L_i=L_i(p,q)$.

\medskip
The rest of the arguments do use linearity of the functions $L_i$.

We have identified four cases of the injective map $\mathfrak F:\;S\times S\to \mathcal K$ to obtain families $\{l_{pq}\}_{(p,q)\in S\times S}$ of lines in $\mathbb{FP}^3$, whose pair-wise intersections are in one-to-one correspondence with the solutions of equations \eqref{dists} (the first two cases), \eqref{degdists} (the third case), \eqref{dirs} (the fourth case). Also, by Lemma \ref{reguli}, the regulus condition (iii) if Theorem \ref{GKT} is automatically satisfied by these families of lines. That remains is to check Conditions (i) and (ii) of Theorem \ref{GKT}. It turns out that these conditions may fail, but ``not too badly'', namely that the situation is nonetheless amenable to the slight generalisation of Theorem \ref{GKT}, Theorem \ref{GKTP}.

\subsection{Checking Conditions (i), (ii) of Theorem \ref{GKT}}
In this section we identify the scenarios under which the concurrency/coplanarity conditions of the family of lines $\{l_{pq}\}$ in $\mathbb{FP}^3$ defined by (\ref{map}) may fail, for all of the four cases above. We also describe their possible failures in terms of the underlying problem in the plane. This having been done, proofs of Theorems \ref{metr}-\ref{directions} will be completed in the next section, after the original plane problems have been restricted to ensure that Conditions (i),(ii) of Theorem \ref{GKT} have been satisfied. In the forthcoming argument we will use the discussion in Section \ref{Bg} about $\alpha$- and $\beta$-planes in $\Ka$, corresponding to concurrency/coplanarity of lines in $\mathbb{FP}^3$.

\subsubsection{Positive definite metric case}
In this case the Pl\"ucker vector $L_{pq}$, assigned to $(p,q)$ via (\ref{map}) is
\begin{equation}\label{mape}
[a\cdot p+\alpha\cdot q: b\cdot p+ \beta\cdot q:1:a\cdot p-\alpha\cdot q:b\cdot p- \beta\cdot q: -(a\cdot p)^2 -(b\cdot p)^2 + (\alpha\cdot q)^2 + (\beta\cdot q)^2].
\end{equation}

Note that in the special case (\ref{esline}) of the Euclidean distance, considered in \cite{GK}, one has equation \eqref{rp}, i.e.
$$
L_{pq} = \left[\frac{q_2-p_2}{2}: \frac{p_1-q_1}{2}:1:\frac{p_2+q_2}{2}:-\frac{p_1+q_1}{2}: \frac{p_1^2+p_2^2-q_1^2-q_2^2}{4}\right].
$$

\begin{lemma}\label{lemerd}
The case satisfies conditions of Theorem \ref{GKT}, given that $-1$ is not a square in $\mathbb{F}$.\end{lemma}
\begin{proof}
Since $a$ is not a multiple of $b$ and $\alpha$ is not a multiple of $\beta$, the linear forms $L_1,\ldots, L_4$ thus defined are linearly independent.

Let us check the concurrency condition (i) of Theorem \ref{GKT}.

Let $\ur=(u_1,u_2,u_3)$, $\om=(a\cdot p+\alpha\cdot q, b\cdot p+ \beta\cdot q,1)$, and
$$\ve =(a\cdot p-\alpha\cdot q, b\cdot p- \beta\cdot q, -(a\cdot p)^2 -(b\cdot p)^2 + (\alpha\cdot q)^2 + (\beta\cdot q)^2).$$
Suppose, we deal with the case of concurrency at the point $\ur\in \mathbb F^3$, that is  $\ve = \ur\times\om$. This means,
$$\begin{aligned}
u_1 &= u_3(a\cdot p+\alpha\cdot q) - b\cdot p + \beta\cdot q,\\
u_2 &= u_3(b\cdot p+\beta\cdot q) + a\cdot p - \alpha\cdot q.\end{aligned}
$$
The matrices multiplying $p$ and $q$ are, respectively  $$\left(
                                                           \begin{array}{c}
                                                             u_3a^T - b^T \\
                                                             u_3b^T + a^T \\
                                                           \end{array}
                                                         \right),\qquad \left(\begin{array}{c}
                                                             u_3\alpha^T + \beta^T \\
                                                             u_3\beta^T - \alpha^T \\
                                                           \end{array}
                                                         \right).$$ These matrices, since $a,b$, as well as $\alpha,\beta$ are linearly independent, are non-degenerate, provided that $-1$ is not a square in $\mathbb{F}$.
                                                         Hence, given $\ur\in \mathbb F^3$, at most one line $l_{pq}$ for each $q$ may be incident to $\ur$, that is Condition (i) is satisfied at $\ur$.
                                                         
                                                       The special case of  concurrency at infinity, would fix the values of $a\cdot p+\alpha\cdot q$ and $b\cdot p+ \beta\cdot q$.  (See the discussion in Section \ref{Bg} concerning $\alpha$-and $\beta$-planes.) The same conclusion then follows  by linear independence of of $a$ and $b$, as well as  $\alpha$ and $\beta$.

Let us check the coplanarity condition (ii) of Theorem \ref{GKT}.
Suppose we are in the generic case of planes with equations $\ur\cdot \qu=-1$, when  $\om = \ur\times\ve$, for some $\ur\in \mathbb F^3$. (Throughout this section the boldface $\boldsymbol q$ denotes a variable in $\mathbb{F}^3$, not to be confused with $q\in S$). This means,
\begin{equation}\label{syst}\begin{array}{ccc}
u_1(b\cdot p-\beta \cdot q) - u_2(a\cdot p-\alpha\cdot q) &= &1,\\
u_2 v_3 - u_3 (b\cdot p-\beta \cdot q)  &=&  a\cdot p + \alpha\cdot q,\\
-u_1v_3 + u_3 (a\cdot p-\alpha\cdot q) & =& b\cdot p+\beta \cdot q.\end{array}
\end{equation}
Suppose, both $u_1,u_2\neq 0$. Then we basically copy the discussion as to the concurrency case. From the last two equations, and using the first one
$$\begin{array}{cc}
u_1(b\cdot p-\beta \cdot q) - u_2(a\cdot p-\alpha\cdot q) &= 1,\\
u_1(a\cdot p + \alpha\cdot q)
+ u_2(b\cdot p+\beta \cdot q)&=-u_3.\end{array}
$$
The matrices multiplying $p$ and $q$ are, respectively
$$ \left(
\begin{array}{cc}u_1b^T - u_2a^T \\
u_1a^T  + u_2b^T
\end{array}
\right), \qquad \left(
\begin{array}{cc} -u_1\beta^T + u_2\alpha^T \\
u_1\alpha^T  + u_2\beta^T
\end{array}
\right), 
$$
Since pairs of vectors $a,b$ and $\alpha, \beta$ are linearly independent, the two matrices are non-degenerate, given that $-1$ is not a square in $\mathbb{F}$. Hence for each $q$, there is a unique $p$, satisfying these equations and vice versa.

Besides, if, say $u_2=0$, then $u_1\neq0$, and the first two equations (\ref{syst}) become
$$
u_1(b\cdot p-\beta \cdot q) =1,\qquad   a\cdot p + \alpha\cdot q =-\frac{u_3}{u_1}.
$$
Hence in both cases, for each variable $q$, there is a unique $p$, satisfying equations (\ref{syst}) and vice versa.

We conclude that whenever $\mathbb{F}$ is such that $-1$ is not a square, that is whenever it is meaningful to speak of positive definite matrices,
Condition (ii) of Theorem \ref{GKT} is satisfied. To be fair, the above consideration has not yet dealt with the special case of co-planarity in the plane through the origin.  (See the discussion in Section \ref{Bg} concerning $\alpha$-and $\beta$-planes.) The latter case fixes the values of $a\cdot p-\alpha\cdot q$ and  $b\cdot p- \beta\cdot q$. The same conclusion follows  by linear independence of $a$ and $b$, as well as  $\alpha$ and $\beta$.
\end{proof}

\subsubsection{Signature $(1,1)$ metric case}
In this case the Pl\"ucker vector $L_{pq}$, assigned to $(p,q)$ via (\ref{map}) is

\begin{equation}\label{mapm}
[a\cdot p+\alpha\cdot q:b\cdot p+ \beta\cdot q:1:a\cdot p-\alpha\cdot q:-b\cdot p+ \beta\cdot q: -(a\cdot p)^2 + (b\cdot p)^2 + (\alpha\cdot q)^2 - (\beta\cdot q)^2].
\end{equation}

\begin{lemma}\label{lemmink}
Conditions (i), (ii)  of Theorem \ref{GKT} may fail at certain points, as well as in certain planes. However, lines $l_{pq}$ and $l_{p'q'}$ are concurrent at such a point or coplanar in such a plane if and only if $$(p-p')^T M_1 (p-p') = (q-q')^T M_2 (q-q')=0.$$\end{lemma}

\begin{proof}
Since $a$ is not a multiple of $b$ and $\alpha$ is not a multiple of $\beta$, the linear forms $L_1,\ldots, L_4$ are linearly independent. We verify conditions (i), (ii) of Theorem  \ref{GKT} only in the case of ``generic'' $\alpha$- and $\beta$-planes as described in Section \ref{Bg}. The special case of concurrency at infinity or co-planarity in a plane $\ur\cdot \qu=0$ through the origin in $\mathbb{F}^3$ follows as in the previous lemma.

Let $\ur=(u_1,u_2,u_3)$, $\om=(a\cdot p+\alpha\cdot q, b\cdot p+ \beta\cdot q,1)$, and
$$\ve =(a\cdot p-\alpha\cdot q,-b\cdot p+ \beta\cdot q, -(a\cdot p)^2 + (b\cdot p)^2 + (\alpha\cdot q)^2 - (\beta\cdot q)^2).$$
Suppose, $\ve = \ur\times\om$. This means,
\begin{equation}\label{check}\begin{aligned}
u_1 &= u_3(a\cdot p+\alpha\cdot q) + b\cdot p - \beta\cdot q,\\
u_2 &= u_3(b\cdot p+\beta\cdot q) + a\cdot p - \alpha\cdot q.\end{aligned}
\end{equation}
The matrices multiplying $p$ and $q$ are, respectively, $$\left(
                                                           \begin{array}{c}
                                                             u_3a^T + b^T \\
                                                             u_3b^T + a^T \\
                                                           \end{array}
                                                         \right), \qquad \left(
                                                           \begin{array}{c}
                                                             u_3\alpha^T - \beta^T \\
                                                             u_3\beta^T - \alpha^T \\
                                                           \end{array}
                                                         \right).$$ These matrices are degenerate if and only if $u_3=\pm 1$. Otherwise, for each $q$, there is a unique $p$, satisfying the concurrency equations and vice versa, i.e. unless $u_3=\pm 1$,  Condition (i)  of Theorem \ref{GKT} is satisfied at $\ur$. 
                                                         
  If $u_3=\pm1$, the equations (\ref{check}) become
$$\begin{aligned}
u_1 &=  (b \pm a)\cdot p + (\pm \alpha-\beta)\cdot q,\\
u_2 &= (a \pm b )\cdot p+(\pm \beta - \alpha)\cdot q.\end{aligned}
$$
Suppose, lines $l_{pq}$ and $l_{p'q'}$ both find themselves concurrent  at such a point $(u_1,u_2,\pm1)$. It follows that
$$(a \pm b) \cdot (p-p') = (\alpha \pm \beta)(q-q') = 0.$$
Thus, by Lemma \ref{iso} the Minkowski distances between $p,p'$, as well as $q,q'$ are zero. The converse is also true by construction and Lemma \ref{iso}: if the latter equation is satisfied, the lines $l_{pq},l_{p'q'}$ are concurrent at a point with $u_3=\pm1$.

Let us check the coplanarity condition (ii) of Theorem \ref{GKT}.
Suppose now, $\om = \ur\times\ve$. This means,
\begin{equation}\label{systm}\begin{array}{ccc}
u_1(-b\cdot p+\beta \cdot q) - u_2(a\cdot p-\alpha\cdot q) &= &1,\\
u_2 v_3 - u_3 (-b\cdot p+\beta \cdot q)  &=&  a\cdot p + \alpha\cdot q,\\
-u_1v_3 + u_3 (a\cdot p-\alpha\cdot q) & =& b\cdot p+\beta \cdot q.\end{array}
\end{equation}
If, say $u_2=0$, then $u_1\neq 0$, and the first two equations (\ref{syst}) become
$$
u_1(-b\cdot p+\beta \cdot q) =1,\qquad a\cdot p + \alpha\cdot q =-\frac{u_3}{u_1}.
$$
Since the pairs of vectors $a,b$ and $\alpha,\beta$ are linearly independent, the coplanarity condition in such a  plane $\ur\cdot \qu = -1$ is satisfied.

Suppose now, both $u_1,u_2\neq 0$. Then we basically copy the discussion as to the concurrency case. Eliminating the term with $v_3$ from the last two equations and using the first one
\begin{equation}\label{sys1}\begin{array}{cc}
-u_1(-b\cdot p+\beta \cdot q) + u_2(a\cdot p-\alpha\cdot q) &= -1,\\
u_1(a\cdot p + \alpha\cdot q)+ u_2(b\cdot p+\beta \cdot q)&=-u_3.\end{array}
\end{equation}
The matrices multiplying $p$ and $q$ are, respectively,
$$ \left(
\begin{array}{cc}u_1b^T + u_2a^T \\
u_1a^T  + u_2b^T
\end{array}
\right), \qquad \left(
\begin{array}{cc} - u_1\beta^T - u_2\alpha^T \\
u_1\alpha^T  + u_2\beta^T
\end{array}
\right).
$$
Hence, the coplanarity condition of Theorem \ref{GKT}  may be violated in  the plane $\ur\cdot \qu = -1$ if only if $u_1=\pm u_2$.

If $u_1 =  \pm u_2\neq 0$ equation \eqref{sys1} become
$$\begin{array}{cc}
(a\pm b) \cdot p - (\alpha\pm\beta) \cdot q &= \frac{1}{u_1},\\
(a\pm b) \cdot p + (\alpha\pm\beta) \cdot q&=-\frac{u_3}{u_1}.\end{array}
$$
Suppose, lines $l_{pq}$ and $l_{p'q'}$ both find themselves in such exceptional plane. It follows that
$$(a\pm b) \cdot (p-p') = (\alpha\pm\beta)(q-q') = 0.$$
Thus, by Lemma \ref{iso} the Minkowski distances between $p,p'$, as well as $q,q'$ are zero. The converse is also true by construction and Lemma \ref{iso}: if the latter equation is satisfied, the lines $l_{pq},l_{p'q'}$ are coplanar in an exceptional plane as above.

\end{proof}

\subsubsection{Degenerate case}
In this case the Pl\"ucker vector, assigned to $(p,q)$ via (\ref{map}) is
\begin{equation}\label{mapd}
L_{pq}= [a\cdot p: \beta\cdot q:1: c\cdot p:\delta\cdot q: -(a\cdot p)(c\cdot p) - (\beta\cdot q)(\delta\cdot q)].
\end{equation}

\begin{lemma}\label{lemdeg}
This case satisfies conditions (i), (ii) of Theorem \ref{GKT}, but for some special points and planes. However, the lines $l_{pq}$ and $l_{p'q'}$ are concurrent at such a point and co-planar in such a plane if and only if $p-p'$ is in a kernel\footnote{We say {\em a} kernel, since it may be the left or right one. Which one -- can bee seen within the proof.} of $M_1$ and $q-q'$ is in a kernel of $M_2$.\end{lemma}
\begin{proof}
Since $a$ is not a multiple of $c$ and $\beta$ is not a multiple of $\delta$, the linear forms $L_1,\ldots, L_4$ are linearly independent.

Let us check the concurrency condition (i).

Let $\ur=(u_1,u_2,u_3)$, $\om=(a\cdot p,\beta\cdot q,1)$, and
$$\ve =(c\cdot p, \delta\cdot q,  -(a\cdot p)(c\cdot p) - (\beta\cdot q)(\delta\cdot q)).$$
Suppose, we are dealing with the generic concurrency case, that is $\ve = \ur\times\om$. This means,
$$\begin{aligned}
u_2 - u_3 \beta\cdot q &= c\cdot p,\\
u_3 a\cdot p - u_1 &= \delta\cdot q.\end{aligned}
$$
There is a unique solution $p$ for every $q$, and vice versa, except when $u_3=0$. In the latter case, given $(u_1,u_2)$, for every $(p,q)$ such that $c\cdot p=u_2,\; \delta\cdot q=-u_1,$ the lines $l_{pq}$ can be concurrent at $(u_1,u_2,0)$. If $l_{pq}$ and $l_{p'q'}$ are concurrent at such point, then $c\cdot (p-p') = \delta\cdot (q-q')=0$, that is the vector $p-p'$ is in the right kernel of $M_1$ and $q-q'$ is in the right kernel of $M_2$. In the case of concurrency at infinity, we fix the values of $a\cdot p$ and $\beta\cdot q$ and therefore come to the same conclusion, only $p-p'$ is now in the left kernel of $M_1$ and $q-q'$  in the left kernel of $M_2$.

Let us check the coplanarity condition (ii). Dealing with the special case of planes $\ur\cdot \qu=0$ through the origin in $\mathbb{F}^3$, we fix the values of $c\cdot p$ and $\delta\cdot q$. Two lines $l_{pq}$ and $l_{p'q'}$ are coplanar in such a plane if and only if $c\cdot (p-p') = \delta\cdot (q-q')=0$, that is $p-p'$ is in the right kernel of $M_1$ and $q-q'$ is in the right kernel of $M_2$. 

In the generic case of planes $\ur\cdot\qu=-1$, suppose $\om = \ur\times\ve.$ This means
$$\begin{aligned}
u_2v_3 - u_3\delta\cdot q &= a\cdot p, \\
u_3c\cdot p - u_1v_3&=\beta\cdot q,\\
u_1 \delta\cdot q - u_2 c\cdot p&=1. \end{aligned}
$$
Suppose $u_1=0$, $u_2\neq 0$. The equations become $c\cdot p = -\frac{1}{u_2}$, $\beta\cdot q = -\frac{u_3}{u_2}$, which means Condition (ii) may fail in the plane
with the equation $u_2x_2+u_3x_3 = -1, \;u_2\neq 0$, and every line $l_{pq}$, such that $c\cdot p = -\frac{1}{u_2}$, $\beta\cdot q = -\frac{u_3}{u_2}$ lies in this plane.

If $l_{pq}$ and $l_{p'q'}$ are coplanar in such plane, then $c\cdot (p-p') = \beta\cdot (q-q')=0$, that is  $p-p'$ is in the right kernel of $M_1$ and $q-q'$ is in the left kernel of $M_2$.

Similarly, we may have exceptional planes with $u_1\neq 0$, $u_2= 0$, in which case lines $l_{pq}$ and $l_{p'q'}$ are coplanar in such plane, if and only if  $a\cdot (p-p') = \delta\cdot (q-q')=0$, that is  $p-p'$ is in the left kernel of $M_1$ and $q-q'$ is in the right kernel of $M_2$.

If both $u_1,u_2\neq 0$ we get

$$\begin{aligned}
u_1a\cdot p+u_2\beta\cdot q + u_1u_3\delta\cdot q- u_2u_3c\cdot p &=0, \\
u_1 \delta\cdot q - u_2 c\cdot p&=1.\end{aligned}
$$
The matrices, multiplying $p$ and $q$ are, respectively
$$\left(\begin{array}{c} u_1a^T -u_2u_3c^T \\ -u_2 c^T\end{array}\right), \qquad  \left(\begin{array}{c} u_2\beta^T +u_1u_3\delta^T \\ u_1 \delta^T\end{array}\right),$$  and are both non-singular, which means, the coplanarity condition of Theorem \ref{GKT} is satisfied.
\end{proof}

\subsubsection{Directions case}
In this final case we deal with Pl\"ucker vectors as follows:
\begin{equation}\label{mapdir}
L_{pq}= [a\cdot p:b\cdot p:1:\lambda b_1q_1+b_2q_2: -\lambda a_1q_1-a_2q_2: (p_1q_2-\lambda p_2q_1)(a_2b_1-a_1b_2)].
\end{equation}

\begin{lemma}\label{lemdir}
The concurrency condition (i) of Theorem \ref{GKT} is satisfied. The co-planarity condition (ii) is satisfied if and only if the point set $S$ has $O(\sqrt{N})$ points on any straight line.\end{lemma}
\begin{proof}
Since $a$ is not a multiple of $b$ and $\lambda, a_2b_1-a_1b_2\neq 0$, the linear forms $L_1,\ldots, L_4$ are linearly independent. Moreover, the concurrency condition at infinity fixes the values of $a\cdot p$
and $b\cdot p$ and is therefore satisfied. 

Let us check the concurrency condition (i) at a point in $\ur \in \mathbb F^3$.

Let $\ur=(u_1,u_2,u_3)$, $\om=(a\cdot p, b\cdot p,1)$, and
$$\ve =(\lambda b_1q_1+b_2q_2, -\lambda a_1q_1-a_2q_2, (p_1q_2-\lambda p_2q_1)(a_2b_1-a_1b_2)).$$
Suppose, $\ve = \ur\times\om$. This means,
$$\begin{aligned}
u_2 - u_3 b\cdot p &= \lambda b_1q_1+b_2q_2,\\
u_3 a\cdot p - u_1 &= -\lambda a_1q_1-a_2q_2.\end{aligned}
$$
The matrix multiplying  $q$ is non-degenerate. Thus Condition (i) of Theorem \ref{GKT}  is satisfied:  given $\ur$ and $p$ there is a unique $q$ satisfying these equations.

Let us check the coplanarity condition (ii). Dealing with planes $\ur\cdot \qu=0$ through the origin in $\mathbb{F}^3$ means fixing  the values $\lambda b_1q_1+b_2q_2$ and $ -\lambda a_1q_1-a_2q_2$, that is fixes $q$. Thus Condition (ii) is satisfied in these planes.

Suppose $\om = \ur\times\ve.$ This means
$$\begin{aligned}
u_2v_3 + u_3(\lambda a_1q_1+a_2q_2) &= a\cdot p,\\
u_3(\lambda b_1q_1+b_2q_2) - u_1v_3&=b\cdot p,\\
-u_1 (\lambda a_1q_1+a_2q_2) - u_2 (\lambda b_1q_1+b_2q_2)&=1. \end{aligned}
$$
Suppose $u_1=0$, $u_2\neq 0$. The equations become $\lambda b_1q_1+b_2q_2 = -\frac{1}{u_2}$, $b\cdot p = -\frac{u_3}{u_2}$, which means Condition (ii) fails in the plane
with the equation $u_2x_2+u_3x_3 = -1, \;u_2\neq 0$, and every line $l_{pq}$, such that $b\cdot p = -\frac{u_3}{u_2}$, $\lambda b_1q_1+b_2q_2 = -\frac{1}{u_2}$ lies in this plane.
The number of such lines $l_{pq}$ will be $O(N)$ if the point set $S$ has $O(\sqrt{N})$ points on each line in the corresponding two families of parallel lines.

Similarly, one deals with the case $u_1=0$, $u_2\neq 0$.

If both $u_1,u_2\neq 0$ we get

$$\begin{aligned}
u_1u_3(\lambda a_1q_1+a_2q_2)-u_1a\cdot p+u_2u_3(\lambda b_1q_1+b_2q_2)-u_2b\cdot p&=0,\\
-u_1 (\lambda a_1q_1+a_2q_2) - u_2 (\lambda b_1q_1+b_2q_2)&=1.\end{aligned}
$$
If $u_3=0$, Condition (ii) fails and any line $l_{pq}$, such that $(u_1a+u_2b)\cdot p=0$ as well as $q_1\lambda(u_1a_1+u_2b_1) + q_2(u_1a_2+ u_2b_2)=-1$. Hence we conclude that Condition (ii) will be satisfied if and only if the point set $S$ has $O(\sqrt{N})$ points on {\em any} straight line.

We finally note that if none of the $u_1,u_2,u_3$ is zero, Condition (ii) is satisfied, for then both $p$ and $q$ in the latter set of two equations are multiplied by non-degenerate matrices.
\end{proof}
\subsection{Conclusion of proofs of Theorems \ref{metr}-\ref{directions}}
Theorem \ref{metr} in the positive definite case and Theorem \ref{directions} follow immediately by Theorem \ref{GKT}, since the families of lines $\{l_{pq}\}$ in $\R^3$, defined via the map \eqref{map} as
\eqref{mape} and \eqref{mapdir} satisfy all its conditions, by Lemmas \ref{reguli}, \ref{lemerd}, and \ref{lemdir}.

As for Theorem \ref{metr} in the signature $(1,1)$ case, as well as Theorem \ref{degmetr} we act as follows. One can choose two subsets $S_1$ and $S_2$ of $S$, with, say at least $\frac{N}{16}$ elements each, such that
for any $(p,p')\in S_1\times S_2$, neither $(p-p')^TM_1(p-p')$, nor $(p-p')^TM_2(p-p')$ equals zero. Define two families of lines
$$L_j= \{l_{pq}\}_{p,q\in S_j},\qquad j=1,2.$$

Apply Theorem \ref{GKTP}, whose conditions are satisfied by Lemmas \ref{lemmink}, \ref{lemdeg}, respectively as to the Minkowski/degenerate cases. It follows that the equation
$$
(p-p')^T M_1 (p-p') = (q-q')^T M_2 (q-q'), \qquad (p,q)\in S_1\times S_1 ,\;(p',q')\in S_2 \times S_2
$$
has $O(N^3\log N)$ solutions. Note that the values of the matrix products involved, by the assumptions on $S_1$, $S_2$ are nonzero.

One can make $O(1)$ choices of the pair of positive proportion subsets $(S^i_1,S^i_2)$ of $S$, with $i=1,\ldots, K=O(1)$, such that whenever
$$
(p-p')^T M_1 (p-p') = (q-q')^T M_2 (q-q')\neq 0, \qquad (p,q)\in S\times S,\;(p',q')\in S\times S,
$$
then for some $i=1,\ldots, K$,
$$
(p-p')^T M_1 (p-p') = (q-q')^T M_2 (q-q'), \qquad (p,q)\in S^i_1\times S^i_1 ,\;(p',q')\in S^i_2 \times S^i_2.
$$
But for the pair of sets $S_1^i$, $S_2^i$ Theorem \ref{GKTP} applies as above, and $K=O(1)$. This completes the proof of Theorems \ref{metr}, \ref{degmetr}. \qed

\section*{Appendix. Derivation of \eqref{rpz} via Elekes-Sharir framework}
Here we show that in the case $S\subset \mathbb S^2$, the line $l_{pq}$, that is the point in the Klein quadric, arising from the condition \eqref{rpz}, is indeed the set of $SO(3)$ symmetries taking $\boldsymbol p$ to $\boldsymbol  q$. We use the Clifford algebra representation of $SO(3)$, whose manifold is $\mathbb{FP}^3$.

Traditionally rotations about a point in three dimensions  were represented by unit quaternions. Clifford algebras  generalise quaternions. See \cite{JS} for their applications in kinematics. The appropriate Clifford algebra to use here is $Cl(3, 0)$. The algebra has $3$ generators 
$e_1$, $e_2$ and $e_3$. These generators anti-commute: $e_ie_j = -e_je_i$ if $i\neq j$ and they all square to $1$.

In this algebra points $\boldsymbol p=(p_1,p_2,p_3)$ on the two-sphere can be represented by grade $1$ elements of the form
\begin{equation}\label{spinr}p=p_1e_1 +p_2e_2 +p_3e_3.\end{equation} (We do not use boldface notation for Clifford algebra elements.)  The Clifford conjugate of such an element is given by
$p^- = -p$, so that $$pp^- = -(p_1^2 + p_2^2 + p_3^2),$$ and for points on the two-sphere this will be constant.

The spin group in this Clifford algebra lies in the even sub-algebra. A general element of ${Spin}(3)$ is given as
\begin{equation}\label{spin}
\tilde g = s_0 + s_1 e_2e_3 + s_2e_1e_3 + s_3e_1e_2,\end{equation}
subject to the relation 
$$
\tilde g\tilde g^- = s_0^2 + s_1^2 + s_2^2 + s_3^2 = 1, 
$$
where the Clifford conjugate on a grade 2 element is given by	$(e_i e_j )^-= -e_i e_j $, $i\neq j$.	The quaternion algebra arises by replacing $e_2e_3\to i$, $e_1e_3\to j,$ $e_1e_2\to k$. The group manifold of $Spin(3)$ is thus the three-sphere $\mathbb S^3.$ The action of this group on points $\boldsymbol p \in \mathbb S^2$ is given by conjugation of the corresponding grade $1$ element $p$  in the Clifford algebra, as follows:
$$g\circ p = \tilde g p \tilde g^-.$$
It is easy to see that this action preserves the square of the distance $pp^-$ of the point from the origin.

The group $Spin(3)$ double covers the rotation group $SO(3)$, for $\tilde g$ and $-\tilde g$  give the same rotation about the origin. To avoid this, we can take the parameters $s_0, . . . , s_3$ in \eqref{spin} as homogeneous coordinates in a 3-dimensional projective space
 $\mathbb{FP}^3$, rather than on $\mathbb S^3$. 
This establishes a one-to-one correspondence between elements of $SO(3)$ and points in $\mathbb{FP}^3$. Notation-wise, to make a difference between $SO(3)$ and its double cover $Spin(3)$,  an element of the group $SO(3)$ will be written in the following as 
$$g = s_0 +s_1e_2e_3 +s_2e_3e_1 +s_3e_1e_2,$$  that is the tilde will be dropped, meaning that the parameters $(s_0:s_1:s_2:s_3)$ are now homogeneous coordinates. So $gg^- \in  \mathbb F \setminus {0}.$ The action of the group on $\mathbb S^2$ must also be changed slightly. Rather than (\ref{spinr}) points on $\mathbb S^2$ will be represented by a quadric in homogeneous coordinates. Consider elements of $\mathbb{FP}^3$ given by  $(p_0 : p _1: p_2 : p_3)$. Now represent these points in the Clifford algebra as
$$p=p_0+p_1e_1 +p_2e_2 +p_3e_3.$$
So $$pp^- = p_0^2 - p_1^2 - p_2^2 - p_3^2,$$ and hence $p_0$ has the meaning of the radius of a sphere,
centred at the origin, given that the corresponding elements $p$ in the Clifford algebra satisfy $pp^- = 0.$ If $g \in SO(3)$ then the action of $g$ on the sphere can still be written as 
$$g\circ p = gpg^-.$$

Now, consider the set of elements of $SO(3)$ which transform a point $\boldsymbol p=(p_1,p_2,p_3)$ to a point $\boldsymbol q=(q_1,q_2,q_3)$ on the two-sphere. The Clifford algebra representations $p$, $q$ will satisfy the latter Clifford algebra equation,
 i.e.,
$$
gpg^- = q\qquad \Rightarrow \qquad gp - qg = 0.$$
This gives four linear equations in the quantities $(s_0:s_1:s_2: s_3)$ by equating the coefficients of the basis elements $e_1, e_2, e_3$ and $e_1e_2e_3.$ However, only two of the equations are independent and hence we have a line of solutions. That is the set of $SO(3)$-elements transforming $\boldsymbol  p$ to $\boldsymbol  q$ on $\mathbb S^2$ is a line in $\mathbb{FP}^3.$

This line can be parameterised in several ways. E.g., any rotation that moves $\boldsymbol  p$ to a non-antipodal $\boldsymbol  q$ can be decomposed as a rotation about $\boldsymbol p$ followed by a rotation by the angle $\pi$ about the line in the plane of $\boldsymbol  p$ and $\boldsymbol  q$ which bisects the two vectors. In the Clifford algebra this can be written as,
\begin{equation}\label{lpqpar} g = [(p_1 + q_1)e_2e_3 + (p_2 + q_2)e_3e_1 + (p_3 + q_3)e_1e_2][c + s(p_1e_2e_3 + p_2e_1e_3 + p_3e_1e_2)],
\end{equation}
 where $c$ and $s$ can be thought of as homogeneous parameters or the first column of a $SO(2)$ matrix, that is $c = \cos \theta/2$ and $s = \sin \theta/2$, $\theta$ being the
angle of rotation about $\boldsymbol  p.$
 
Passing from equation (\ref{lpqpar}) to Pl\"ucker coordinates is a short calculation by formula \eqref{Pc}, where the two points on the line one uses for passing to Pl\"ucker coordinates via \eqref{Pc}  can be taken, for instance, as $(c,s)=(1,0)$ and $(0,1)$. The result is precisely \eqref{rpz}, 
where the common factor, 
$$(p_1^2 +p_2^2 +p_3^2)+(p_1q_1 +p_2q_2 +p_3q_3) = (q_1^2 +q^2_2 +q_3^2)+(p_1q_1 +p_2q_2 +p_3q_3)$$
has been cancelled from homogeneous Pl\"ucker coordinates. We leave the special case of antipodal $\boldsymbol  p$ and $\boldsymbol  q$ to the reader.

\hspace{2cm}


\begin{thebibliography}{4}



\bibitem{CFKP} J.W. Cannon, W.J. Floyd,  R. Kenyon and W.R. Parry. {\em Hyperbolic Geometry.} In {\em Flavors of Geometry,} MSRI Publications, Volume {\bf 31}, 1997, 59--115.

\bibitem{ES} G. Elekes, M. Sharir. {\it Incidences in three dimensions and distinct distances in the
    plane}. Proceedings 26th ACM Symposium on Computational Geometry (2010), 413--422.

\bibitem{Fo} K. Ford. {\it The distribution of integers with a divisor in a given interval.} Annals of Math., {\bf 168} (2008), 367--433.

\bibitem{GK} L. Guth, N. H. Katz. {\it On the Erd\H os distinct distance problem in the plane}.  Ann. of Math. (2) {\bf 181} (2015), no. 1, 155--190.

\bibitem{MRS} B. Murphy, O. Roche-Newton, I.D. Shkredov. {\em Variations on the sum-product problem}. SIAM Journal on Discrete Mathematics {\bf 29}(1) (2015), 514--540.


\bibitem{PW} H. Pottmann and J. Wallner. {\em Computational line geometry.} Paperback edition. Mathematics and Visualization. Springer-Verlag, Berlin, 2010. 563 pp.

\bibitem{RN} O. Roche-Newton. {\em A Short proof of a near-optimal cardinality estimate for the product of a sum set. } 31st International Symposium on Computational Geometry, 74--80, LIPIcs. Leibniz Int. Proc. Inform., 34, Schloss Dagstuhl. Leibniz-Zent. Inform., Wadern, 2015.

\bibitem{RR} O. Roche-Newton and M. Rudnev. {\em On the Minkowski distances and products of sum sets.}  Israel J. Math. {\bf 209}(2015), no. 2, 507--526. 





\bibitem{JS} J.M. Selig. {\em Geometric Fundamentals of Robotics.} Monographs in Computer Science. Springer, 2007, 416 pp.

\bibitem{U} P. Ungar. {\em 2N Noncollinear points determine at least 2N directions.} J. Combin. Th. A
{\bf 33}, no 3 (1982), 343--347.




\end{thebibliography}
\end{document}